\documentclass[12pt,leqno]{amsart}

\usepackage{amsmath,amssymb,amsthm}
\usepackage{hyperref}

\newcommand{\T}{{\sf T}}
\newcommand{\Tg}{{\sf T_{\bullet}}}

\newcommand{\Id}{\mathrm{Id}}

\DeclareMathOperator{\diag}{diag}

\newtheorem{theorem}{\sc Theorem}[section]
\newtheorem{thm}[theorem]{\sc Theorem}
\newtheorem{lem}[theorem]{\sc Lemma}
\newtheorem{prop}[theorem]{\sc Proposition}

\newtheorem{conjecture}[theorem]{\sc Conjecture}

\newtheorem*{thmA}{Theorem A}
\newtheorem*{thmB}{Theorem B}

\theoremstyle{definition}
\newtheorem{exe}[theorem]{\sc Example}

\newtheorem{rem}[theorem]{\sc Remark}

\begin{document}

\title[Generalized torsion in triangular matrix groups]{Generalized torsion in triangular matrix groups}
\author[Kato]{Júlia Kato}
\address{Departamento de Matem\'atica, Universidade de Bras\'ilia,
Bras\'ilia-DF, 70910-900 Brazil}
\email{(Kato) juliakato@mat.unb.br}
\author[Silva]{Douglas Silva}
\address{Departamento de Matem\'atica, Universidade Federal de Minas Gerais, Belo Horizonte-MG, 31270-901 Brazil}
\email{(Silva) douglasvps@ufmg.br}

\begin{abstract}
An element of a group is called generalized torsion if a finite product of its conjugates equals the identity. We characterize the generalized torsion subset of upper triangular matrix groups over commutative rings and determine when these groups satisfy positive generalized identities in terms of torsion and exponent properties of the unit group of the underlying ring. Moreover, as an application of the developed theory, we construct triangular matrix groups satisfying positive generalized identities of prescribed degrees and obtain lower bounds for the minimum degrees of such identities. We also provide a counterexample to a question of Sherman~\cite{Sherman} concerning groups generated by elements of infinite order.

\end{abstract}

\maketitle

\section{Introduction}

An element $g$ of a group $G$ is called \emph{generalized torsion} if there exist elements $x_1,\dots,x_k\in G$ such that
\[
g^{x_1}g^{x_2}\cdots g^{x_k}=1,
\qquad
g^x=x^{-1}gx.
\]
The smallest such integer $k$ is called the \emph{generalized order} of $g$.

Generalized torsion has been studied in several contexts, notably as an obstruction to bi-orderability \cite{Rhemtulla,Bludov}, and more recently in several classes of groups from a topological perspective \cite{Teragaito2016,Juan,ItoMotegiTeragaito2021,ItoMotegiTeragaito2023}.

The set of all generalized torsion elements of $G$ is denoted by $\Tg(G)$. In general, this set need not be a subgroup and is not necessarily closed under taking powers \cite{BM}. Consequently, determining $\Tg(G)$ is often a difficult problem.

Moreover, the size of $\Tg(G)$ can vary dramatically. At one extreme, $\Tg(G)=\{1\}$ for torsion-free nilpotent groups, free groups, and Thompson's group $F$ (see \cite{BSS, DeroinNavasRivas}). At the opposite extreme, when $\Tg(G)=G$, we say that $G$ is a \emph{generalized torsion group}. Examples of such groups arise in a variety of settings \cite{Gorjuskin,Gorcakov,Osin}.

Even within linear groups, generalized torsion can exhibit very different behaviors. Indeed, Witte~\cite{W98} proved that $SL_n(K)$ is a generalized torsion group whenever $n>2$ or $|K|>3$, whereas if the additive group of a ring $R$ is torsion-free (for example, if $R$ is an integral domain of characteristic zero) then $UT_n(R)$ is torsion-free nilpotent and hence has no nontrivial generalized torsion. This suggests that linear groups provide a rich setting in which to study generalized torsion.

Closely related to generalized torsion are \emph{positive generalized identities}. Following Abdollahi, Daoud, Farrokhi and Guerboussa \cite{ADFG}, a polynomial word in one variable with coefficients in $G$ is an expression of the form
\[
w(x)=x^{g_1}x^{g_2}\cdots x^{g_k},
\]
where $g_1,\dots,g_k\in G$ are fixed. The integer $k$ is called the \emph{degree} of $w$.
We say that $G$ satisfies the positive generalized identity $w$ if
\[
w(g)=1
\qquad\text{for all } g\in G.
\]

Positive generalized identities have been studied in several works \cite{Endimioni2006,ADFG,BM,BS}. In the context of solvable groups, their existence often imposes strong finiteness conditions. For instance, \cite{ADFG} showed that if a finitely generated solvable group satisfies a positive generalized identity of prime degree $p$, then it must be a finite $p$-group. Thus, the possible identity degrees are often highly rigid.

These contrasting phenomena motivate the study of generalized torsion and positive generalized identities in solvable linear groups.

In this paper, we study the solvable linear group of upper triangular matrices
\[
T_n(R)=UT_n(R)\rtimes D_n(R^\times)
\]
over a commutative ring with identity $R$.

Our first main result gives a complete description of the generalized torsion elements of $T_n(R)$, showing that they are entirely determined by the torsion of the unit group $R^\times$.

\begin{thmA}\phantomsection\label{thm:A}
Let $R$ be a commutative ring with identity and $n \geq 1$. Then
\[
\Tg(T_n(R))
=
UT_n(R)\rtimes D_n(\T(R^\times)),
\]
where $\T(R^\times)$ denotes the torsion subgroup of $R^\times$. 

In particular:
\begin{enumerate}
    \item $\Tg(T_n(R))$ is a characteristic subgroup of $T_n(R)$; 
    \item $UT_n(R) \subseteq \Tg(T_n(R))$; and
    \item $\Tg(T_n(R)) = T_n(R)$ if and only if $R^\times$ is a torsion group.
\end{enumerate}
\end{thmA}

In fact, we prove that $\Tg(T_n(R))$ is itself a generalized torsion group, providing, for every commutative ring $R$, a linear nilpotent-by-abelian generalized torsion group. Moreover, the same techniques allow us to determine precisely when $\Tg(T_n(R))$, as well as certain related intermediate subgroups, satisfies a positive generalized identity.

\begin{thmB}\phantomsection\label{thm:B}
Let $R$ be a commutative ring with identity and $n \geq 1$. Then $T_n(R)$ satisfies a positive generalized identity if and only if $R^\times$ has finite exponent.
\end{thmB}

A noteworthy feature of \hyperref[thm:B]{Theorem~B} is that the possible identity degrees are determined by suitable torsion elements of $R^\times$. 

More generally, the key tool underlying both theorems is that diagonal conjugation acts on $UT_n(R)$ by rescaling superdiagonal entries. Combined with the nilpotent structure of $UT_n(R)$, this reduces questions about generalized torsion and positive generalized identities in $T_n(R)$ to the arithmetic properties of the unit group $R^\times$.

We illustrate this dependence on the base ring by providing examples of triangular matrix groups with markedly different behavior. For instance, if $R$ is the ring of integers of an algebraic number field, then $T_n(R)$ satisfies a positive generalized identity if and only if the underlying field is either $\mathbb{Q}$ or an imaginary quadratic field. Likewise, if $R$ is the integral group ring of a finite abelian group $G$, then $T_n(R)$ satisfies a positive generalized identity if and only if $G$ has exponent $1$, $2$, $3$, $4$, or $6$.

Moreover, as applications of the developed theory, we construct triangular matrix groups satisfying positive generalized identities of new minimum degrees, extending the constructions of~\cite{ADFG,BM}, and provide a counterexample to a question of Sherman~\cite{Sherman} concerning groups generated by elements of infinite order.

The paper is organized as follows. Section~\ref{sec:prelims} establishes notation and background on triangular matrix groups and positive generalized identities. Section~\ref{sec:thmA} contains the proof of \hyperref[thm:A]{Theorem~A}. Section~\ref{sec:identities} treats positive generalized identities and contains the proof
of \hyperref[thm:B]{Theorem~B} and the results on intermediate subgroups.
Section~\ref{sec:examples} presents families of rings illustrating Theorems~A and~B. Section~\ref{sec:applications} contains the applications to metabelian groups and
the Sherman counterexample.

\section{Preliminaries}\label{sec:prelims}
Throughout the paper, $R$ denotes a nonzero commutative ring with identity, and
$R^\times$ denotes its group of units. We consider the triangular matrix group
\[
T_n(R)=UT_n(R)\rtimes D_n(R^\times).
\]

For a subgroup $S\leq R^\times$, we define
\[
G_n(S)=\pi^{-1}(D_n(S)),
\]
where $\pi:T_n(R)\to D_n(R^\times)$ is the canonical projection onto the diagonal subgroup. Explicitly,
\[
G_n(S)=
\left\{
\begin{pmatrix}
d_1 & a_{1,2} & \cdots & a_{1,n}\\
0 & d_2 & \cdots & a_{2,n}\\
\vdots & \vdots & \ddots & \vdots\\
0 & 0 & \cdots & d_n
\end{pmatrix}
\;\middle|\;
d_i\in S,\;
a_{i,j}\in R
\right\}.
\]

Thus, the groups $G_n(S)$ interpolate between the unitriangular subgroup and the full triangular group:
\[
UT_n(R)=G_n(\{1\}),
\qquad
T_n(R)=G_n(R^\times).
\]

The group $UT_n(R)$ is nilpotent of class $n-1$, and its lower central series is given by
\[
\gamma_k(UT_n(R))
=
\left\{
M\in UT_n(R)
\mid
M_{i,j}=0
\text{ whenever }
0<j-i<k
\right\}.
\]
In other words, $\gamma_k(UT_n(R))$ consists of the matrices whose first $k-1$ superdiagonals vanish.

\medskip

In what follows, we will prove not only that $UT_n(R)$ is always a generalized torsion subset of $T_n(R)$, but also that one can choose a single polynomial word in $T_n(R)$ that vanishes on all unitriangular matrices.

Following \cite{ADFG}, this motivates the definition of a positive generalized identity for a subset $N$ of a group $G$. We say that $N$ satisfies a positive generalized identity in $G$ if there exists a polynomial word
\[
w(x)=x^{g_1}x^{g_2}\cdots x^{g_n},
\]
with $g_i\in G$, such that
\[
w(g)=1
\qquad \text{for all } g\in N.
\]

Finally, following the notation introduced in \cite{Endimioni2006}, we denote by $\widehat{\mathcal B}_n$ the class of groups satisfying a positive generalized identity of degree $n$.

\medskip

The following lemma shows that the multiplication of elements in $\gamma_k(UT_n(R))$ behaves additively on a fixed superdiagonal. This observation will be used to construct positive generalized identities in $UT_n(R)$.

These results are well known, but we include proofs for the reader's convenience.

\begin{lem}\label{lem:additivity}
Let $M_1,\dots,M_r \in \gamma_k(UT_n(R))$. Then, for every $j \le n-k$,
\[
\left(\prod_{i=1}^r M_i\right)_{j,j+k}
=
\sum_{i=1}^r (M_i)_{j,j+k}.
\]
\end{lem}

\begin{proof}
It suffices to prove the case $r=2$, since the general case follows by induction.

Let $A,B \in \gamma_k(UT_n(R))$, and fix $j \le n-k$. We compute
\[
(AB)_{j,j+k} = \sum_{\ell=j}^{j+k} A_{j,\ell} B_{\ell,j+k}.
\]

If $j<\ell<j+k$, then $\ell-j<k$, so $A_{j,\ell}=0$. Hence only the terms $\ell=j$ and $\ell=j+k$ contribute, and
\[
(AB)_{j,j+k}
=
A_{j,j} B_{j,j+k} + A_{j,j+k} B_{j+k,j+k}.
\]

Since both matrices are unitriangular, $A_{j,j}=B_{j+k,j+k}=1$, and therefore
\[
(AB)_{j,j+k} = A_{j,j+k} + B_{j,j+k}.
\]
\end{proof}

In order to exploit this, we will consider products of conjugates by successive powers of a \emph{single} diagonal matrix. The key point is that conjugation by 
such a matrix rescales the entries on a fixed superdiagonal by powers of a single ratio, so that the resulting sum can be arranged to vanish. 
This will allow us to eliminate one superdiagonal at a time.

\begin{lem}\label{lem:conjugation-sum}
Let \(1\leq k\leq n-1\), let
\(
M\in\gamma_k(UT_n(R)),
\)
and let
\(
D=\diag(d_1,\dots,d_n)\in D_n(R^\times).
\)
Then, for every integer \(r\geq 0\) and every \(1\leq j\leq n-k\),

\[
\left(MM^D\cdots M^{D^r}\right)_{j,j+k}
=
\left(
1+\frac{d_{j+k}}{d_j}
+\cdots+
\left(\frac{d_{j+k}}{d_j}\right)^r
\right)M_{j,j+k}.
\]
\end{lem}

\begin{proof}
Since \(\gamma_k(UT_n(R))\) is normal, each \(M^{D^\ell}\) lies in
\(\gamma_k(UT_n(R))\). Applying Lemma~\ref{lem:additivity} to the
product
\[
M\,M^D\cdots M^{D^r},
\]
we obtain
\[
\left(M\,M^D\cdots M^{D^r}\right)_{j,j+k}
=
\sum_{\ell=0}^{r}\left(M^{D^\ell}\right)_{j,j+k}.
\]

Using
\[
\left(M^{D^\ell}\right)_{j,j+k}
=
\left(\frac{d_{j+k}}{d_j}\right)^\ell M_{j,j+k},
\]
the result follows.
\end{proof}

\section{Proof of Theorem A}\label{sec:thmA}
First, we show that the subgroup $UT_n(R)$ satisfies a positive generalized identity in $T_n(R)$. For rings $R$ of positive characteristic, this follows from the fact that $UT_n(R)$ is a torsion group of finite exponent, as we will see next.

\begin{prop}\label{prop:char-finite}
Let $R$ be a ring with $\operatorname{char} R = s > 0$, then $UT_n(R)$ satisfies a positive generalized identity of degree $s^{n-1}$ in $T_n(R)$.
\end{prop}
\begin{proof}
By Lemma~\ref{lem:conjugation-sum}, for any
$M \in \gamma_i(UT_n(R))$, we have
\[
(M^{s})_{j,j+i}
=
(\underbrace{1+\dots+1}_{s})\, M_{j,j+i}
=
0,
\]
and therefore
\[
M^{s}\in \gamma_{i+1}(UT_n(R)).
\]

Since $UT_n(R)$ is nilpotent of class $n - 1$, $M^{s^{n-1}} = \Id_n$. Thus $UT_n(R)$ is a torsion group whose exponent divides $s^{n-1}$. Consequently, $UT_n(R)$ satisfies a positive generalized identity of degree $s^{n-1}$ in $T_n(R)$.
\end{proof}

Independently of the characteristic, we can still obtain a positive generalized identity whenever there exist an integer \(s\geq2\) and a unit \(\zeta\in R^\times\) such that
$1 + \zeta + \cdots + \zeta^{s-1} = 0$.

\begin{lem}\label{lem:TG-UT}
Let $R$ be a commutative ring, let \(s\geq2\) be an integer, and let $\zeta \in R^\times$ satisfy
\[
1 + \zeta + \cdots + \zeta^{s-1} = 0.
\]
Then $UT_n(R)$ satisfies a positive generalized identity of degree $s^{\,n-1}$ inside $G_n(\langle \zeta \rangle) \subseteq T_n(R)$.
\end{lem}

\begin{proof}
For each $1 \leq i \leq n-1$, define a diagonal matrix
\[
D_i = \diag(d^{(i)}_1,\dots,d^{(i)}_n)
 \in G_n(\langle \zeta \rangle)\]
whose diagonal entries are constant on blocks of length $i$, alternating between $\zeta$ and $1$, that is,
\[ D_i = \diag
(\underbrace{\zeta,\dots,\zeta}_{i},
 \underbrace{1,\dots,1}_{i},
 \underbrace{\zeta,\dots,\zeta}_{i},
 \underbrace{1,\dots,1}_{i},
 \dots).
\]
With this choice, for every $j \le n-i$, the indices $j$ and $j+i$ lie in consecutive blocks, and therefore
\[
\frac{d^{(i)}_{j+i}}{d^{(i)}_j} \in \{\zeta,\zeta^{-1}\}.
\]

For each $i \ge 1$, define recursively a word $w_i(x)$ by
\[
w_1(x) = x,
\qquad
w_{i+1}(x) = w_i(x)\, w_i(x)^{D_i} \cdots w_i(x)^{D_i^{\,s-1}}.
\]
We will show that for every $M \in UT_n(R)$, one has $w_n(M) = 1$, which yields a positive generalized identity of degree $s^{\,n-1}$ in $G_n(\langle \zeta\rangle)$.

We prove by induction that $w_i(M) \in \gamma_i(UT_n(R))$ for all $i$.

For $i=1$, this is clear since $\gamma_1(UT_n(R)) = UT_n(R)$. Assume that $w_i(M) \in \gamma_i(UT_n(R))$. Applying Lemma~\ref{lem:conjugation-sum}, for $j \le n-i$ we obtain
\[
(w_{i+1}(M))_{j,j+i}
=
\left(
1 + \frac{d^{(i)}_{j+i}}{d^{(i)}_j}
+ \cdots +
\left(\frac{d^{(i)}_{j+i}}{d^{(i)}_j}\right)^{s-1}
\right)
(w_i(M))_{j,j+i}.
\]

Since $\frac{d^{(i)}_{j+i}}{d^{(i)}_j} \in \{\zeta,\zeta^{-1}\}$ and both satisfy
\[
1 + \zeta^{\pm1} + \cdots + (\zeta^{\pm1})^{s-1} = 0,
\]
it follows that $(w_{i+1}(M))_{j,j+i} = 0$ for all $j \le n-i$. Since $w_{i+1}(M)$ is a product of conjugates of $w_i(M)$, it lies in $\gamma_i(UT_n(R))$, and the vanishing of the $i$-th superdiagonal implies that
\[
w_{i+1}(M) \in \gamma_{i+1}(UT_n(R)).
\]

Thus $w_i(M) \in \gamma_i(UT_n(R))$ for all $i$. Since $UT_n(R)$ is nilpotent of class $n-1$, we obtain $w_n(M) = 1$.

By construction, $w_n$ has degree $s^{\,n-1}$, which proves the claim.
\end{proof}

\begin{rem}
Notice that $-1 \in R^\times$ always satisfies $1+(-1)=0$.
Therefore, taking $\zeta=-1$, Lemma~\ref{lem:TG-UT} shows that $UT_n(R)$ always satisfies a positive generalized identity of degree $2^{n-1}$ inside $T_n(R)$.

However, in general, $UT_n(R)$ may satisfy positive generalized identities of different degrees in $T_n(R)$, as in the case of rings with $\operatorname{char}(R)\neq 2$ considered in Proposition~\ref{prop:char-finite}. We will exploit positive generalized identities whose degrees are not powers of two in Section~\ref{sec:applications}. When $\operatorname{char}(R)=2$, the identity obtained above coincides precisely with the one given in Proposition~\ref{prop:char-finite}. 
\end{rem}

Now we proceed with the proof of \hyperref[thm:A]{Theorem~A}.

\begin{proof}[Proof of Theorem~A]
Let $M \in \Tg(T_n(R))$. Then $\pi(M)$ is a generalized torsion element of $D_n(R^\times)$. Since
\[
D_n(R^\times) \simeq (R^\times)^n
\]
is abelian, every generalized torsion element has finite order. Therefore, each entry of $\pi(M)$ has finite order, and hence $\pi(M) \in D_n(\T(R^\times))$.

Consequently,
\[
\Tg(T_n(R))
\subseteq
UT_n(R)\rtimes D_n(\T(R^\times)).
\]

Conversely, let
\[
M \in UT_n(R)\rtimes D_n(\T(R^\times)).
\]
Then $\pi(M) \in D_n(\T(R^\times))$, and therefore $\pi(M)$ has finite order. Let $k \ge 1$ be such that $\pi(M)^k = \Id_n$. It follows that $M^k \in UT_n(R)$. Hence, it suffices to show that every element of $UT_n(R)$ is generalized torsion in $T_n(R)$, since the generalized torsion subset is closed under taking roots.

By Lemma~\ref{lem:TG-UT} (with $\zeta=-1$ and $s=2$), the group $UT_n(R)$ satisfies a positive generalized identity of degree $2^{n-1}$ inside $T_n(R)$. Hence,
\[
UT_n(R) \subseteq \Tg(T_n(R)).
\]

This completes the proof.
\end{proof}

\begin{rem}
As observed in the proof of \hyperref[thm:A]{Theorem~A}, the subgroup
\[
\Tg(T_n(R)) = G_n(\T(R^\times))
\]
is itself a generalized torsion group. Indeed, for every element of this subgroup, the generalized torsion identities constructed in the proof use only conjugation by diagonal matrices in $D_n(\T(R^\times))$.
\end{rem}

\section{Positive generalized identities in triangular matrix groups}\label{sec:identities}

In the previous section, we showed that $UT_n(R)$ always satisfies a positive generalized identity inside $T_n(R)$, and that every generalized torsion element of $T_n(R)$ lies in
\[
G_n(\T(R^\times)) = UT_n(R) \rtimes D_n\big(\T(R^\times)\big).
\]

We now study when the subgroups $G_n(S)$ satisfy positive generalized identities.

\begin{thm}\label{thm:G(S)-identity}
Let $R$ be a commutative ring, let \(s\geq2\) be an integer, and let $S \leq R^\times$ be a subgroup containing an element $\zeta \in S$ such that
\[
1+\zeta+\zeta^2+\cdots+\zeta^{s-1}=0.
\]
Then $G_n(S)$ satisfies a positive generalized identity if and only if $\exp(S)<\infty$. In this case,
\[
G_n(S)\in \widehat{\mathcal B}_{s^{\,n-1}\exp(S)}.
\]
\end{thm}

\begin{proof}
Since $\zeta \in S$, Lemma~\ref{lem:TG-UT} implies that $UT_n(R)$ satisfies a positive generalized identity of degree $s^{\,n-1}$ inside $G_n(S)$.

Assume first that $\exp(S)=k<\infty$. Let $M\in G_n(S)$. Since $\pi(M)$ is an element of $D_n(S) \simeq S^n$, we have $M^k\in UT_n(R)$. Applying the positive generalized identity for $UT_n(R)$ to $M^k$, we obtain
\[
w_n(M^k)=1,
\]
where $w_n$ is the word constructed in Lemma~\ref{lem:TG-UT}. Equivalently, defining $w_0(x)=x^k$, we have
\[
w_n(w_0(M))=1
\qquad \text{for all } M\in G_n(S).
\]
Hence $G_n(S)$ satisfies a positive generalized identity of degree $s^{\,n-1}k$.

Conversely, suppose that \(G_n(S)\) satisfies
\[
x^{g_1}\cdots x^{g_m}=1.
\]
For \(a\in S\), let \(D(a)=\diag(a,\ldots,a)\).
Since \(D(a)\in Z(G_n(S))\), substitution gives
\[
1=D(a)^{g_1}\cdots D(a)^{g_m}
 =D(a)^m.
\]
Thus \(a^m=1\) for every \(a\in S\), and hence
\(\exp(S)\mid m\).
\end{proof}

As a direct consequence, we have the following proof for \hyperref[thm:B]{Theorem~B}.

\begin{proof}[Proof of Theorem~B]
Apply Theorem~\ref{thm:G(S)-identity} with $S=R^\times$, $\zeta=-1$, and $s=2$.
\end{proof}

We end this section with an example where the generalized order of elements in $T_n(R)$ need not be the upper bound obtained in Theorem~\ref{thm:G(S)-identity}.

\begin{rem}
By Theorem~\ref{thm:G(S)-identity}, the group $T_2(\mathbb Z)$ satisfies a positive generalized identity of degree~$4$. However, every nontrivial element has generalized order~$2$.

Indeed, for
\[
M=
\begin{pmatrix}
\varepsilon_1 & x\\
0 & \varepsilon_2
\end{pmatrix},
\qquad
X=
\begin{pmatrix}
1 & \varepsilon_1x\\
0 & -1
\end{pmatrix},
\]
one has $X^{-1}MX=M^{-1}$, and hence $MM^X=1$.
\end{rem}

\section{Examples}\label{sec:examples}

In this section, we present examples of classes of rings for which $G_n(\T(R^\times))$ and/or $T_n(R)$ satisfy a positive generalized identity by applying \hyperref[thm:B]{Theorem~B} and Theorem \ref{thm:G(S)-identity}.

We begin with two classes of commutative rings for which $G_n(\T(R^\times))$ always satisfies a positive generalized identity, whereas $T_n(R)$ does so only in specific, classifiable cases.

\subsection{Rings of integers of algebraic number fields}

Let $K$ be an algebraic number field and let $O_K$ denote its ring of integers. Let $r_1$ be the number of real embeddings of $K$, and let $r_2$ be the number of pairs of conjugate complex embeddings.

By Dirichlet's Unit Theorem (see \cite{Neukirch}), one has
\[
O_K^\times \cong \mu_K \times \mathbb Z^r,
\]
where $\mu_K$ is the group of roots of unity in $K$ and $r = r_1 + r_2 - 1.$

In particular, $\T(O_K^\times)$ is finite. Hence, by Theorem~\ref{thm:G(S)-identity}, we obtain the following.

\begin{exe}
The group $G_n(\T(O_K^\times))$ satisfies a positive generalized identity.
\end{exe}

Moreover, $O_K^\times = \T(O_K^\times)$ if and only if $r=0$. Therefore, by \hyperref[thm:B]{Theorem~B}, we have the following characterization.

\begin{exe}
The group $T_n(O_K)$ satisfies a positive generalized identity if and only if $K=\mathbb Q$ or $K$ is an imaginary quadratic field.
\end{exe}

\begin{proof}
If $r=0$, then $r_1+r_2=1$. Hence either $r_1=1$ and $r_2=0$, or $r_1=0$ and $r_2=1$.

In the first case, $K=\mathbb Q$. In the second case, $[K:\mathbb Q]=2$, and $K$ has no real embeddings. Therefore, $K$ is an imaginary quadratic field.
\end{proof}

\subsection{Integral group rings over finite abelian groups}

Let $G$ be a finite abelian group. Denote by $n_2$ the number of elements of order $2$ in $G$, and by $c$ the number of cyclic subgroups of $G$.

A theorem of Higman (see \cite{Higman1940}) states the following.

\begin{thm}[Higman]
Let $G$ be a finite abelian group. Then
\[
\mathbb ZG^\times \cong \pm G \times \mathbb Z^r,
\]
where
\[
r = \frac{|G|+1+n_2-2c}{2}.
\]
\end{thm}

In particular, $\T(\mathbb ZG^\times)$ is finite. Hence, by Theorem~\ref{thm:G(S)-identity}, we obtain the following.

\begin{exe}
Let $G$ be a finite abelian group. Then the group $G_n(\T(\mathbb Z G^\times))$ satisfies a positive generalized identity.
\end{exe}

Moreover, $\mathbb ZG^\times=\T(\mathbb ZG^\times)$ if and only if $r=0$. Therefore, by \hyperref[thm:B]{Theorem~B}, we obtain the following characterization.

\begin{exe}
Let $G$ be a finite abelian group. Then $T_n(\mathbb ZG)$ satisfies a positive generalized identity if and only if $G$ has exponent $1$, $2$, $3$, $4$, or $6$.
\end{exe}
\begin{proof}
Let $a_d$ be the number of cyclic subgroups of order $d$. Since each cyclic subgroup of order $d$ has $\varphi(d)$ generators,
\[
|G|
=
1+n_2+\sum_{d\ge3}\varphi(d)a_d, \quad
c=1+n_2+\sum_{d\ge3}a_d.
\]
Hence
\[
|G|+n_2+1-2c
=
\sum_{d\ge3}(\varphi(d)-2)a_d.
\]

Now $\varphi(d)-2=0$ exactly for $d=3,4,6$, and $\varphi(d)-2>0
$ for all $d \geq 3$ with $d\notin\{3,4,6\}$. Therefore, $r = 0$ if and only if every cyclic subgroup of $G$ has order $1,2,3,4,$ or $6$, which is equivalent to the condition of $G$ having an exponent equal to $1,2, 3, 4,$ or $6$.
\end{proof}

If $R$ is a field, then $\T(R^\times)$ has finite exponent if and only if $\T(R^\times)$ is finite. We conclude this section by presenting two classes of fields illustrating when this phenomenon does and does not occur.

\subsection{Orderable fields}

Let $(K,\leq)$ be an orderable field. Then
\[
\T(K^\times)=\{-1,1\}.
\]

Indeed, if $x>1$, then clearly $x\notin \T(K^\times)$. Likewise, if $0<x<1$, then $x\notin \T(K^\times)$. Hence, if $x\in \T(K^\times)$, then either $x=1$ or $x<0$.

Suppose now that $x<0$. If $n$ is odd, then $x^n<0$, so $x^n\neq1$. Therefore, if $x$ has finite order, this order must be even. Hence, there exists $m\geq1$ such that
\[
x^{2m}=1.
\]
Thus $(x^2)^m=1$. Since $x^2>0$, the only positive torsion element is $1$, and therefore $x^2=1$. Consequently, $x=\pm1$.

Therefore, $\T(K^\times)$ is finite, while $K^\times$ is not. Hence, we obtain the following.

\begin{exe}
    Let $K$ be an orderable field. Then $G_n(\T(K^\times))$ satisfies a positive generalized identity, while $T_n(K)$ does not.
\end{exe}

\subsection{Algebraically closed fields}

Let $K$ be an algebraically closed field. 

Let \(p\) range over primes different from
\(\operatorname{char}(K)\). Since \(K\) is algebraically closed,
\(x^p-1\) splits in \(K\); since \(p\neq\operatorname{char}(K)\),
it has \(p\) distinct roots. Therefore it has a nontrivial root,
and every nontrivial \(p\)-th root of unity has order \(p\).
Thus \(K^\times\) contains torsion elements of arbitrarily large
order, so \(\exp(\T(K^\times))=\infty\).

Hence, we have the following:

\begin{exe}
    Let $K$ be an algebraically closed field. Then neither $G_n(\T(K^\times))$ nor $T_n(K)$ satisfy a positive generalized identity.
\end{exe}

\section{Applications}\label{sec:applications}

In \cite{ADFG}, the authors study the existence of infinite finitely generated solvable groups $G$ in the class $\widehat{\mathcal{B}}_n$. They show that no such groups exist when $n$ is prime, while examples are constructed when $n$ is a square. In \cite{BM}, a family of examples is obtained for every integer $n$ divisible by $p^3$, for a given prime $p$. We extend these results while imposing stronger structural conditions on the groups.

\begin{thm}
Let $p$ be a prime and let $n>0$ be such that $p^2 \mid n$. Then there exists an infinite finitely presented metabelian group in $\widehat{\mathcal{B}}_n$.
\end{thm}

\begin{proof}
Write $n = p^2 m$, and let $R = \mathbb{Z}[\zeta_{pm}]$ be the ring of $pm$-th cyclotomic integers. Let $\zeta_p \in R$ be a primitive $p$-th root of unity in $\langle \zeta_{pm} \rangle$, and consider
\[
G_2(\langle \zeta_{pm} \rangle)
=
\left\{
\begin{pmatrix}
d_1 & x \\
0 & d_2
\end{pmatrix}
\in T_2(R)
\;\middle|\;
d_i \in \langle \zeta_{pm} \rangle,\; x \in R
\right\}.
\]

Then $G_2(\langle \zeta_{pm} \rangle)$ is an infinite, finitely presented metabelian group. Moreover, since $\zeta_p \in \langle \zeta_{pm} \rangle$ and $\exp(\langle \zeta_{pm} \rangle)=pm$, Theorem~\ref{thm:G(S)-identity} applies, and we obtain
\[
G_2(\langle \zeta_{pm} \rangle) \in \widehat{\mathcal{B}}_{p^2 m} = \widehat{\mathcal{B}}_{n}.
\]
\end{proof}

The examples in \cite{ADFG} satisfying a positive generalized identity of degree \(p^2\) also satisfy one of degree \(n\) whenever \(p^2 \mid n\), simply by repeating the identity \(n/p^2\) times. However, the minimum degree of a positive generalized identity for each of those groups is therefore at most \(p^2\). 

By contrast, the groups
constructed above satisfy a positive generalized identity of degree \(n\), while every such identity has degree divisible by \(n/p\).  Consequently, their minimum positive generalized identity degree is at least \(n/p\). 

Thus, whereas the earlier construction yields identities of arbitrarily large degrees by taking multiples of a fixed degree-\(p^2\) identity, our construction yields groups whose minimum possible identity degree grows with \(n\).

\begin{conjecture}
There are no infinite finitely generated solvable groups $G$ satisfying a positive generalized identity of degree $n$, where $n$ is a square-free integer.
\end{conjecture}

Throughout this paper, we have constructed positive generalized identities using conjugating elements of finite order. This fact relates to the work of G.\,J. Sherman~\cite{Sherman}, who studied groups whose torsion subset forms a subgroup. In particular, he proved the following.

\begin{prop}[Sherman]
Suppose that a group \(G\) contains an element \(g\) such that, for
every positive integer \(k\) and every choice of torsion elements
\(x_1,\dots,x_k\in G\), one has
\[
g^{x_1}g^{x_2}\cdots g^{x_k}\neq1.
\]
Then \(G\) is generated by its elements of infinite order.
\end{prop}

G.\,J. Sherman asked whether the converse holds. We provide a counterexample to this question.

\begin{thm}
There exists a group \(G\), generated by elements of infinite order,
for which there exist a positive integer \(k\) and fixed torsion
elements \(x_1,\dots,x_k\in G\) such that
\[
g^{x_1}g^{x_2}\cdots g^{x_k}=1
\qquad
\text{for every }g\in G.
\]
\end{thm}

\begin{proof}
Let $R = \mathbb{Z} \times \mathbb{Z}$. Then $R^\times = \{\pm1\} \times \{\pm1\}$, so $\T(R^\times)=R^\times$. By \hyperref[thm:B]{Theorem~B}, the group $T_2(R)$ satisfies a positive generalized identity.

More precisely, taking
\[
D = \begin{pmatrix}
-1 & 0 \\
0 & 1
\end{pmatrix},
\]
one obtains a positive generalized identity of the form
\[
w(x) = x^{2} (x^{2})^{D},
\]
where the conjugating element $D$ has finite order. Thus $T_2(R)$ satisfies a positive generalized identity in which all conjugating elements are torsion.

Now consider the subset
\[
S =
\left\{
\begin{pmatrix}
u & x \\
0 & v
\end{pmatrix}
\in T_2(R)
\;\middle|\;
ux + vx \neq 0
\right\}.
\]

We claim that every element of $S$ has infinite order. Since $D_2(R^\times)$ has exponent $2$ and $UT_2(R)$ is torsion-free, an element $M \in T_2(R)$ has finite order if and only if $M^2=\Id$. For
\[
M =
\begin{pmatrix}
u & x \\
0 & v
\end{pmatrix},
\]
one computes
\[
M^2 =
\begin{pmatrix}
u^2 & ux + vx \\
0 & v^2
\end{pmatrix}.
\]
Thus $M^2 = I$ if and only if $ux + vx = 0$. By definition of $S$, this never occurs, and hence every element of $S$ has infinite order.

We now show that $S$ generates $T_2(R)$. Every nontrivial unitriangular matrix lies in $S$, so it remains to generate $D_2(R^\times)$.

Let $u,v \in R^\times$ and choose $y \in R^\times \setminus \{-u,-v\}$. Then
\[
\begin{pmatrix}
1 & -1 \\
0 & vy
\end{pmatrix},
\qquad
\begin{pmatrix}
u & y \\
0 & y
\end{pmatrix}
\in S,
\]
and a direct computation shows that
\[
\begin{pmatrix}
u & 0 \\
0 & v
\end{pmatrix}
=
\begin{pmatrix}
1 & -1 \\
0 & vy
\end{pmatrix}
\begin{pmatrix}
u & y \\
0 & y
\end{pmatrix}.
\]
Thus $D_2(R^\times) \subseteq \langle S \rangle$, and therefore $S$ generates $T_2(R)$.

Hence $T_2(R)$ is generated by elements of infinite order and satisfies a positive generalized identity with torsion conjugating elements.
\end{proof}

\section*{Acknowledgements}
The first author acknowledges financial support from the Fundação de Apoio à Pesquisa do Distrito Federal (FAPDF), project no.~00193.00001270/2025-48. The second author acknowledges financial support from the Fundação de Amparo à Pesquisa do Estado de Minas Gerais (FAPEMIG) through a doctoral scholarship (no.~30500) and project APQ-03491-25.

\end{document}